\documentclass[smallextended,natbib,runningheads]{svjour3}
\journalname{Annals of the Institute of Statistical Mathematics}
\smartqed  
\usepackage{graphicx}
 \usepackage{marvosym}
 \usepackage{mathptmx}
 \usepackage{times,bm}
 \usepackage{bbold}
 \usepackage[LY1]{fontenc}
 
\usepackage{latexsym}
\usepackage{amsmath}
\usepackage{amssymb}
\usepackage{dsfont}
\usepackage[dvipsnames,usenames]{color}
\RequirePackage[colorlinks,citecolor=Blue,linkcolor=blue,urlcolor=red]{hyperref}
\definecolor{Blue}{rgb}{0,0,.4}
\newtheorem{theo}{Theorem}[section]
\newtheorem{lem}{Lemma}[section]
\newtheorem{prop}{Proposition}[section]
\newtheorem{rem}{Remark}[section]
\newtheorem{defi}{Definition}[section]
\newtheorem{cor}{Corollary}[section]

\renewcommand{\thetable}{\thesection.\arabic{table}}

\mathchardef \al   ="010B
  \renewcommand{\alpha}{\textrm{\scalebox{.883}{$\al$}}}
\mathchardef \bet  ="010C
  \renewcommand{\beta}{\textrm{\scalebox{.882}{$\bet$}}}
\mathchardef \gam  ="010D
  \renewcommand{\gamma}{\textrm{\scalebox{.883}{$\gam$}}}
\mathchardef \del  ="010E
  \renewcommand{\delta}{\textrm{\scalebox{.882}{$\del$}}}
\mathchardef \eps  ="010F
  \renewcommand{\epsilon}{\textrm{\scalebox{.883}{$\eps$}}}
\mathchardef \zet  ="0110
  \renewcommand{\zeta}{\textrm{\scalebox{.882}{$\zet$}}}
\mathchardef \et   ="0111
  \renewcommand{\eta}{\textrm{\scalebox{.883}{$\et$}}}
\mathchardef \thet ="0112
  \renewcommand{\theta}{\textrm{\scalebox{.93}{$\thet$}}}
\mathchardef \iot  ="0113
  \renewcommand{\iota}{\textrm{\scalebox{.883}{$\iot$}}}
\mathchardef \kapp ="0114
  \renewcommand{\kappa}{\textrm{\scalebox{.883}{$\kapp$}}}
\mathchardef \lambd="0115
  \renewcommand{\lambda}{\textrm{\scalebox{.883}{$\lambd$}}}
\mathchardef \m    ="0116
  \renewcommand{\mu}{\textrm{\scalebox{.882}{$\m$}}}
\mathchardef \n    ="0117
  \renewcommand{\nu}{\textrm{\scalebox{.883}{$\n$}}}
\mathchardef \ksi  ="0118
  \renewcommand{\xi}{\textrm{\scalebox{.882}{$\ksi$}}}
\mathchardef \pii  ="0119
  \renewcommand{\pi}{\textrm{\scalebox{.883}{$\pii$}}}
\mathchardef \ro   ="011A
  \renewcommand{\rho}{\textrm{\scalebox{.883}{$\ro$}}}
\mathchardef \sig  ="011B
  \renewcommand{\sigma}{\textrm{\scalebox{.883}{$\sig$}}}
\mathchardef \ta   ="011C
  \renewcommand{\tau}{\textrm{\scalebox{.883}{$\ta$}}}
\mathchardef \upsil="011D
  \renewcommand{\upsilon}{\textrm{\scalebox{.883}{$\upsil$}}}
\mathchardef \fii  ="011E
  \renewcommand{\phi}{\textrm{\scalebox{.882}{$\fii$}}}
\mathchardef \xii  ="011F
  \renewcommand{\chi}{\textrm{\scalebox{.883}{$\xii$}}}
\mathchardef \psii ="0120
  \renewcommand{\psi}{\textrm{\scalebox{.883}{$\psii$}}}
\mathchardef \omeg ="0121
  \renewcommand{\omega}{\textrm{\scalebox{.883}{$\omeg$}}}
\mathchardef \vthet="0123
  \renewcommand{\vartheta}{\textrm{\scalebox{.93}{$\vthet$}}}
\mathchardef \vpi  ="0125
  \renewcommand{\varpi}{\textrm{\scalebox{.883}{$\vpi$}}}
\mathchardef \vsig  ="0126
  \renewcommand{\varsigma}{\textrm{\scalebox{.883}{$\vsig$}}}
  \mathchardef \vphi  ="0127
  \renewcommand{\varphi}{\textrm{\scalebox{.882}{$\vphi$}}}
\newcommand{\R}{\mathds{R}}
\newcommand{\N}{\mathds{N}}
\newcommand{\E}{\operatorname{\mathds{E}}}
\newcommand{\Var}{\operatorname{\mathsf{Var}}}
\newcommand{\CH}{\mathcal{H}}

\newcommand{\ds}{\displaystyle}
\newcommand{\ud}{\hspace*{.2ex}\textrm{\rm d}}
\newcommand{\IP}{\mathrm{IP}}

\newcommand{\invsubset}{\mathrel{\reflectbox{\rotatebox[origin=c]{90}{$\subseteq$}}}}
\let\originalleft\left
\let\originalright\right
\renewcommand{\left}{\mathopen{}\mathclose\bgroup\originalleft}
\renewcommand{\right}{\aftergroup\egroup\originalright}

\begin{document}

\title{Unified extension of variance bounds for integrated Pearson family}

\titlerunning{Variance bounds for integrated Pearson family}        

\author{G. Afendras}


\institute{G. Afendras
           \at
              Department of Mathematics, Section of Statistics and O.R.\\
              University of Athens\\
              Panepistemiopolis, 157 84 Athens, Greece\\
              \email{g\_afendras@math.uoa.gr}
           }

\date{}

\maketitle

\begin{abstract}
We use some properties of orthogonal polynomials to provide a class of upper/lower variance bounds for a function $g(X)$ of an absolutely continuous \mbox{random} variable $X$, in terms of the derivatives of $g$ up to some order. The new bounds are better than the existing ones.
\end{abstract}

\section{Introduction}\label{sec1}
\setcounter{equation}{0}
 Let $Z$ be a standard normal random variable and $g:\R\to\R$
 any absolutely continuous (a.c.) function with derivative $g'$.
 \cite{Cher}, using Hermite polynomials, proved that
 (see also the previous papers by
 \cite{Nash}, \cite{BL})
 \[
 \Var g(Z)\leq \E\big(g'(Z)\big)^2,
 \]
 provided that $\E\big(g'(Z)\big)^2<\infty$, where the
 equality holds if and only if $g$
 is a polynomial of degree at most one --
 a linear function. This inequality plays an important role in
 the isoperimetric problem and has been extended and generalized
 by several authors; see, e.g., \cite{Chen,CP1,Pap1,HK,PP,P-Rao} and references therein.

 The results of the present paper are related to the following class of random variables [cf. \cite{Kor,DZ,Joh}; see \cite{APP,AP1,AP2}].

 \begin{defi}\label{defi IP(mu,q)}
 {\rm [Integrated Pearson Family]\quad}
 Let $X$ be a random variable with density $f$ and finite mean $\mu=\E X$. We say that $X$ (or its density) belongs to the integrated Pearson family (or integrated Pearson system) if there exists a quadratic polynomial $q(x)=\delta x^2+\beta x+\gamma$ (with $\delta,\beta,\gamma\in\R$, $|\delta|+|\beta|+|\gamma|>0$) such that
 \begin{equation}\label{qua1}
 \int_{-\infty}^{x}(\mu-t)f(t)\ud{t}=q(x)f(x) \quad \text{for all} \quad x\in\R.
 \end{equation}
 This fact will be denoted by
 $X$ or $f\sim\IP(\mu;q)$ or, more explicitly, $X$ or $f\sim\IP(\mu;\delta,\beta,\gamma)$.
 \end{defi}

 The definition of this class is sometimes considered as equivalent to the Pearson family; cf. \cite{Kor,Joh}. However, this is not precise. In fact, several properties holding for integrated random variables, are not true for all Pearson distributions. For instance, Properties ${\rm P}_3$ and ${\rm P}_4$ in \cite{Ord}, pp. 4--5, are not informative for the behavior of moments (see (1.7)), unless the distribution is integrated Pearson. The same is true for eq. (12.45), p. 22, of \cite{JKB}. In a review paper by \cite{DZ}, the classification of Pearson distributions were related to orthogonal polynomials (see Table 2, p. 296). This implicitly stated family is close to what we call ``integrated Pearson family''. Its properties have been analyzed in detail in a recent work by \cite{AP1}.

 Let a random variable $X\sim\IP(\mu;q)$ and consider a suitable function $g$. \cite{Joh} established Poincar\'e-type upper/lower bounds for the variance of $g(X)$ of the form
 \begin{equation}\label{existing bounds}
 (-1)^{n}\big(\Var{g(X)}-{S}_{n}\big)\geq0,
 \quad \textrm{where} \quad
 {S}_n=\sum_{k=1}^{n}(-1)^{k-1}\frac{\E q^{k}(X)\big(g^{(k)}(X)\big)^2}{k!\prod_{j=0}^{k-1}(1-j\delta)}.
 \end{equation}
 In particular, for the normal see \cite{Pap1} and \cite{HK}; for the gamma see \cite{Pap1}.

 \cite{APP}, using Bessel's inequality, showed that
 \begin{equation}\label{Bessel}
 \Var g(X)\ge\sum_{k=1}^{n}\frac{\E^2q^k(X)g^{(k)}(X)}{k!\E q^k(X)\prod_{j=k-1}^{2k-2}(1-j\delta)};
 \end{equation}
 for the case $n=1$ see \cite{Cac}.

 \cite{AP2} showed that, under appropriate conditions, the following two forms of Chernoff-type upper bounds of the variance of $g(X)$ are valid:
 \begin{equation}\label{S_n upper strong}
 S_{n,\rm(str)}=\sum_{i=1}^{n}\frac{\E^2q^i(X)g^{(i)}(X)}{i!\E{q^i(X)}\prod_{j=i-1}^{2i-2}(1-j\delta)}
               +\frac{\E q^n(X)\big(g^{(n)}(X)\big)^2-\frac{\E^2q^n(X)g^{(n)}(X)}{\E{q^n(X)}}}{(n+1)!\prod_{j=n}^{2n-1}(1-j\delta)},
 \end{equation}
 \begin{equation}\label{S_n upper weak}
 S_{n,\rm(weak)}=\sum_{i=1}^{n-1}\frac{\E^2q^i(X)g^{(i)}(X)}{i!\E{q^i(X)}\prod_{j=i-1}^{2i-2}(1-j\delta)}
              +\frac{\E q^n(X)\big(g^{(n)}(X)\big)^2}{n!\prod_{j=n-1}^{2n-2}(1-j\delta)}.
 \end{equation}
 The equality in (\ref{S_n upper strong}) holds when $g$ is a polynomial of degree at most $n+1$ and in (\ref{S_n upper weak}) when $g$ is a polynomial of degree at most $n$. The bound (\ref{S_n upper weak}) for beta distributions has been shown by \cite{WZ}, using Jacobi polynomials.
 \medskip

 In the present article we shall apply a technique introduced by \cite{AP2} in order to obtain a general class of bounds. Specifically, in Section \ref{sec2} we provide a new class of upper/lower bounds for the variance of $g(X)$. They can be called as ``Poincar\'e-type'' of order $n$ and with point balance $m$. They hold for a subfamily of Pearson distributions. In particular, the bound for $N(\mu,\sigma^2)$ distribution, namely
 \[
 S_{m,n}(g)=\sum_{i=1}^{m}\frac{{m\choose{i}}\sigma^{2i}}{(m+n)_i}\E^2g^{(i)}(X)+
            \sum_{i=1}^{n}(-1)^{i-1}\frac{{n\choose{i}}\sigma^{2i}}{(m+n)_i}\E\big(g^{(i)}(X)\big)^2
 \]
 (for $(x)_k$ see Definition \ref{defi (x)_k}), satisfies the inequality
 \[
 (-1)^n\big(\Var{g(X)}-S_{m,n}(g)\big)\ge0,
 \]
 where the equality holds if and only if $g$ is a polynomial of degree at most $m+n$.
 \medskip

 For fixed $n$, Section \ref{sec3} investigates the bounds $S_{m,n}(g)$ as $m$ increases. It is shown that the bound $S_{m+1,n}(g)$ is better than $S_{m,n}(g)$, i.e.,
 \[
 \big|\Var{g(X)}-S_{m+1,n}(g)\big|\le\big|\Var{g(X)}-S_{m,n}(g)\big|.
 \]
 Also, for any suitable function $g$, $S_{m,n}(g)\to\Var{g(X)}$ as $m\to\infty$.
 \medskip


 \section{Unified extension of variance bounds}
 \label{sec2}
 \setcounter{equation}{0}
 This section presents a wide class of variance bounds. First we prove the following useful lemma.
 \begin{lem}
 \label{lem E q^m(X) |g^(m)(X)|<00 => E q^i(X) |g^(i)(X)|<00}
 Let $X\sim\IP(\mu;q)\equiv\IP(\mu;\delta,\beta,\gamma)$ and consider a positive integer $m$ with $\E X^{2m}<\infty$. Suppose that the function $g$ is defined on the support $J=(\alpha,\omega)$ of $X$, and assume that $g\in C^{m-1}(J)$ and $g^{(m-1)}:=\frac{\ud^{m-1}g(x)}{\ud{x}^{m-1}}$ is absolutely continuous with (a.s.) derivative $g^{(m)}$. If $\E q^m(X)|g^{(m)}(X)|<\infty$ then
 \[
 \E q^i(X)|g^{(i)}(X)|<\infty \ \ \ \textrm{for all} \ \ i=0,1,\ldots,m-1.
 \]
 \end{lem}
 \begin{proof}
 Fix $i\in\{0,1,\ldots,m-1\}$ and assume that $\E q^{i+1}(X)|g^{(i+1)}(X)|<\infty$. Setting $h:=g^{(i)}$ we have that $\E q^{i+1}(X)|h'(X)|<\infty$. Consider the random variable $X_i$ with density $f_i=q^if/\E{q^i(X)}\sim\IP(\mu_i;q_i)$, where $\mu_i=(\mu+i\beta)/(1-2i\delta)$, $q_i=q/(1-2i\delta)$ and $J(X_i)=J$ (see Appendix \ref{appendix properties}). One can easily see that $\E q_i(X_i)|h'(X_i)|<\infty$. Since $\E X^{2m}<\infty$ we get $\E X_i^{2(m-i)}<\infty$. Hence $\E X_i^{2}<\infty$ because $m-i\ge1$. Using Lemma \ref{lem E q^i|g^(i)|} (for $k=1$) we have that $\E |P_{1,i}(X_i)h(X_i)|<\infty$, where $P_{1,i}(x)=x-\mu_i$ is the Rodrigues polynomial of degree $1$ corresponding to the density $f_i$. Since $\mu_i\in(\alpha,\omega)$, we can find $\epsilon>0$ such that $[\mu_i-\epsilon,\mu_i+\epsilon]\subset(\alpha,\omega)$. Thus,
 \medskip

 \noindent
 $\E |P_{1,i}(X_i)h(X_i)|=$

 \noindent
 $=\int_{\alpha}^{\mu_i-\epsilon}(\mu_i-x)|h(x)|f_i(x)\ud{x}
   +\int_{\mu_i-\epsilon}^{\mu_i+\epsilon}|(\mu_i-x)h(x)|f_i(x)\ud{x}
   +\int_{\mu_i+\epsilon}^{\omega}(x-\mu_i)|h(x)|f_i(x)\ud{x}$

 \noindent
 $\ge\epsilon\int_{\alpha}^{\mu_i-\epsilon}|h(x)|f_i(x)\ud{x}
   +\int_{\mu_i-\epsilon}^{\mu_i+\epsilon}|(\mu_i-x)h(x)|f_i(x)\ud{x}
   +\epsilon\int_{\mu_i+\epsilon}^{\omega}|h(x)|f_i(x)\ud{x}$.
 \medskip

 \noindent
 Hence, $\int_{\alpha}^{\mu_i-\epsilon}|h(x)|f_i(x)\ud{x}$ and $\int_{\mu_i+\epsilon}^{\omega}|h(x)|f_i(x)\ud{x}$ are finite. The function $h$ is continuous in the compact interval $[\mu_i-\epsilon,\mu_i+\epsilon]$ so $\int_{\mu_i-\epsilon}^{\mu_i+\epsilon}|h(x)|f_i(x)\ud{x}$ is finite. Therefore, $\E|h(X_i)|=\E|g^{(i)}(X_i)|=\frac{\E{q^i(X)}|g^{(i)}(X)|}{\E{q^i(X)}}$ is finite.

 \noindent
 We have shown that $\E q^{i+1}(X)g^{(i+1)}(X)<\infty$ implies $\E q^{i}(X)g^{(i)}(X)<\infty$. Applying this for $i=m-1,m-2,\ldots,0$, the proof is completed.
 \end{proof}

 Now we give the following definitions that will be used in the sequel.
 \begin{defi}\label{defi (x)_k}
 For $x\in\R$ and $k\in\N$ define:
 \smallskip

 {\rm(a)} $(x)_k=x(x-1)\cdots(x-k+1)$, with $(x)_0=1$.
 \smallskip

 {\rm(b)} $[x]_k=x(x+1)\cdots(x+k-1)$, with $[x]_0=1$.
 \smallskip

 \noindent
 Note that $[x]_k=(-1)^k(-x)_k=(x+k-1)_k$.
 \end{defi}
 \begin{defi}\label{defi H^m,n(X)}
 {\rm[cf. \cite{AP2}]}\quad
 Assume that $X\sim\IP(\mu;q)$ and denote by $q(x)=\delta x^2+\beta x+\gamma$ its quadratic polynomial. Let $J(X)=(\alpha;\omega)$ be the support of $X$ and fix the non-negative integers $m,n$ such that $\E|X|^{2\ell}$ is finite, where $\ell=\max\{m,n\}$. We shall denote by $\CH^{m,n}(X)$ the class of Borel functions $g:(\alpha,\omega)\to\R$ satisfying the following properties:
 \begin{itemize}
  \item[]\begin{itemize}
          \item[$\rm H_1:$]
          $g\in C^{\ell-1}(\alpha,\omega)$ and the function $g^{(\ell-1)}(x):=\frac{\ud^{\ell-1}g(x)}{\ud{x}^{\ell-1}}$ is a.c. in $(\alpha,\omega)$ with a.s.\ derivative $g^{(\ell)}$.
         \end{itemize}
 \item[]\begin{itemize}
          \item[$\rm H_2:$]
          $\E q^n(X)\big(g^{(n)}(X)\big)^2<\infty$ and $\E q^m(X)|g^{(m)}(X)|<\infty$.
         \end{itemize}
 \end{itemize}
 \end{defi}
 Note that from (\ref{E q^n(X)g^(n)(X)<inf => E q^i(X)g^(i)(X)<inf}), Lemma \ref{lem E q^m(X) |g^(m)(X)|<00 => E q^i(X) |g^(i)(X)|<00} and $\E^2{q^i(X)}|g^{(i)}(X)|\le\E{q^i(X)}\cdot\E{q^i(X)}\big(g^{(i)}(X)\big)^2$, $i=1,2,\ldots,n$, if $m\le n$ and if $\E q^n(X)\big(g^{(n)}(X)\big)^2$ is finite then it is implied that $\E q^m(X)|g^{(m)}(X)|$ is finite. For $m=n=0$, the property $\rm H_1$ does not impose any restrictions on $g$, and
 \[
 \CH^{0,0}(X)\equiv L^2(\R,X):=\big\{g:(\alpha,\omega)\to\R \ \textrm{such that} \Var g(X)<\infty\big\}.
 \]
 Also, it is obvious that $\CH^{0,n}=\CH^{1,n}=\cdots=\CH^{n,n}$.
 \newline
 Furthermore, we shall denote by $\CH^{\infty,n}(X)$ and $\CH^{\infty}(X)\equiv\CH^{m,\infty}(X)$ [$m$ is arbitrary because in this case this index is insignificant] the classes of functions $\CH^{\infty,n}(X):=\cap_{m=0}^{\infty}\CH^{m,n}(X)$ and $\CH^{\infty}(X):=\cap_{n=0}^{\infty}\CH^{\infty,n}(X)$; i.e,
 \[
 \begin{split}
 &\CH^{\infty,n}(X)=\big\{g\in{C}^{\infty}(J): \E{q^n(X)}\big(g^{(n)}(X)\big)^2\!<\infty \ \textrm{and} \ \E{q^i(X)}|g^{(i)}(X)|<\infty \ \forall i>n\big\},\\
 &\CH^{\infty}(X)=\big\{g\in{C}^{\infty}(J): \E{q^n(X)}\big(g^{(n)}(X)\big)^2<\infty \ \forall n\in\N\big\}.
 \end{split}
 \]

From Lemma \ref{lem E q^m(X) |g^(m)(X)|<00 => E q^i(X) |g^(i)(X)|<00} and from (\ref{E q^n(X)g^(n)(X)<inf => E q^i(X)g^(i)(X)<inf}),
we conclude that the (finite or infinite) sequence $\CH^{m,n}(X)$ is decreasing in $m$ and in $n$. In particular, if all moments of $X$ exist then
\[
\begin{array}{cccccccccc}
  L^2(\R,X)\equiv & \CH^{0,0}(X) &           &              &           &              &           &        &           &                 \\
[-.5ex]
                  &  \invsubset  &           &              &           &              &           &        &           &                 \\
                  & \CH^{1,0}(X) & \supseteq & \CH^{1,1}(X) &           &              &           &        &           &                 \\
[-.5ex]
                  &  \invsubset  &           & \invsubset   &           &              &           &        &           &                 \\
                  & \CH^{2,0}(X) & \supseteq & \CH^{2,1}(X) & \supseteq & \CH^{2,2}(X) &           &        &           &                 \\
[-.5ex]
                  &  \invsubset  &           & \invsubset   &           & \invsubset   &           &        &           &                 \\
[-1.5ex]
                  &  \vdots      &           & \vdots       &           & \vdots       &           &        &           &                 \\
[-.5ex]
                  &  \invsubset  &           & \invsubset   &           & \invsubset   &           &        &           &                 \\
                  &\CH^{\infty,0}(X)&\supseteq&\CH^{\infty,1}(X)&\supseteq&\CH^{\infty,2}(X)&\supseteq&\cdots&\supseteq & \CH^{\infty}(X).\\
\end{array}
\]

 Let $X\sim\IP(\mu;\delta,\beta,\gamma)$ with $\delta\le0$. Also, consider two (fixed) non-negative integers $m,n$, with $n\ne0$, and a function $g\in\CH^{m,n}(X)$. According to Parseval's identity we have that
 \begin{equation}\label{Var g(X)}
 \Var{g(X)}=\sum_{k=1}^{\infty}c_k^2,
 \end{equation}
 where $c_k=\E g(X)\varphi_k(X)$ are the Fourier coefficients of $g$ with respect to the corresponding (to $X$) orthonormal polynomial system $\{\varphi_k\}_{k=0}^\infty$.

 \noindent
 For each $i=1,2,\ldots,n$ the function $g^{(i)}\in\CH^{m-i,n-i}(X_i)$ and, from Parceval's identity again, $\E\big(g^{(i)}(X_i)\big)^2=\sum_{k=0}^{\infty}\big(c^{(i)}_k\big)^2$,
 where $c^{(i)}_k=\E g^{(i)}(X_i)\varphi_{k,i}(X_i)$ are the Fourier coefficients of $g^{(i)}$ with respect to the orthonormal polynomial system $\{\varphi_{k,i}\}_{k=0}^\infty$ corresponding to $X_i\sim f_i\propto q^if$; see Appendix \ref{appendix properties}. Using (\ref{E varphi_k,i(X_i)g^(i)(X_i)}) we have that
 \begin{equation}\label{E(g^(i)(X))^2}
 {\E{q^{i}(X)\big(g^{(i)}(X)\big)^2}}=\sum_{k=i}^{\infty}\Big((k)_i{\prod_{j=k-1}^{k+i-2}(1-j\delta)}\Big)c_k^2,
 \quad i=1,2,\ldots,n,
 \end{equation}
 see \cite[Lemma 3.1, eq. (3.4)]{AP2}, where each coefficient of $c_k^2$ is positive. Also, from (\ref{E q^k(X)g^(k)(X)}),
 \begin{equation}\label{Eg^(i)(X)}
 {\E{q^{i}(X)g^{(i)}(X)}}=\Big(i!\E{q^i(X)}\prod_{j=i-1}^{2i-2}(1-j\delta)\Big)^{1/2}c_i,
 \quad i=1,2,\ldots,m,
 \end{equation}
 see \cite[Section 3, eq.'s (3.2) and (3.5)]{APP}.

 \noindent
 Let $\bm{\lambda}_n=(\lambda_{1;n},\lambda_{2;n},\ldots,\lambda_{n;n})^{\rm t}\in\R^n$. According to Tonelli's Theorem we have that
 \linebreak
 $\sum_{i=1}^n\sum_{k=i}^\infty|\lambda_{i;n}(k)_i\prod_{j=k-1}^{k+i-2}(1-j\delta)|c_k^2
 =\sum_{i=1}^n|\lambda_{i;n}|\sum_{k=i}^\infty\big[(k)_i\prod_{j=k-1}^{k+i-2}(1-j\delta)\big]c_k^2=$
 \linebreak
 $\sum_{i=1}^n|\lambda_{i;n}|\E{q^i(X)}\big(g^{(i)}\big)^2<\infty$
 and, using Fubini's Theorem,
 \begin{equation}\label{sum_i=1^n lamdba_i,n Eq^i(X)g^(i)(X)/Eq^i(X)}
 \!\sum_{i=1}^{n}\lambda_{i;n}{\E{q^{i}(X)\big(g^{(i)}(X)\big)^2}}\!\!=\!\sum_{k=1}^{\infty}\rho_{k;n}c_k^2,
 \ \textrm{where} \
 \rho_{k;n}=\!\!\!\sum_{i=1}^{\min\{k,n\}}\!\!\!\lambda_{i;n}(k)_i\!\prod_{j=k-1}^{k+i-2}\!(1-j\delta).
 \end{equation}
 We seek a vector $\bm{\lambda}_{m,n}$ such that $\rho_{m+1,n}=\rho_{m+2,n}=\cdots=\rho_{m+n,n}=1$. From (\ref{E(g^(i)(X))^2}) we obtain the system of equations
 \begin{equation}\label{system of equations}
 A_{m,n}\cdot\bm{\lambda}_{m,n}=\mathbf{1}_n,
 \end{equation}
 where the matrix $A_{m,n}\in\R^{n\times{n}}$ has $(r,c)$-element which is given by
 \medskip

 \centerline{$a_{r,c;m,n}=(m+r)_c\prod_{m+r-1}^{m+r+c-2}(1-j\delta)$}
 \medskip

 \noindent
 and the vector $\mathbf{1}_n=(1,1,\ldots,1)^{\rm t}\in\R^n$.

 \noindent
 The above system has the unique solution, see Appendix \ref{appendix system},
 \begin{equation}\label{lambda_i,n;0}
 \mbox{$\lambda_{i;m,n}={(-1)^{i-1}{n\choose{i}}}\big/\big[{(m+n)_i\prod_{j=m}^{m+i-1}(1-j\delta)}\big],\quad i=1,2,\ldots,n.$}
 \end{equation}
 From (\ref{sum_i=1^n lamdba_i,n Eq^i(X)g^(i)(X)/Eq^i(X)}) and (\ref{lambda_i,n;0}), using the hypergeometric series (\ref{Hypergeometric}), we have that $\rho_{k;m,n}=1-(m+n-k)_n\prod_{j=m+k}^{m+n+k-1}(1-j\delta)\big/\big[(m+n)_n\prod_{j=m}^{m+n-1}(1-j\delta)\big]$. Equivalently,
 \[
 \rho_{k;m,n}=\left\{\begin{array}{l@{\hspace{0ex}}c@{\hspace{1.5ex}}r}
                1-\frac{(m+n-k)_n\prod_{j=m+k}^{m+n+k-1}(1-j\delta)}{(m+n)_n\prod_{j=m}^{m+n-1}(1-j\delta)} & , & 1\le k \le m,\hspace{3.45ex}\\
                1                     & , & m< k \le m+n, \\
                1+(-1)^{n-1}\frac{(k-m-1)_n\prod_{j=m+k}^{m+n+k-1}(1-j\delta)}{(m+n)_n\prod_{j=m}^{m+n-1}(1-j\delta)} & , & k > m+n.
              \end{array}\right.
 \]
 Thus,
 \begin{equation}\label{linear combination}
 \begin{split}
  \sum_{i=1}^{n}(-1)&^{i-1}\frac{{n\choose{i}}\E{q^{i}(X)\big(g^{(i)}(X)\big)^2}}{(m+n)_i\prod_{j=m}^{m+i-1}(1-j\delta)}\\
  &=\Var{g(X)}-\sum_{k=1}^{n}\frac{(m+n-k)_n\prod_{j=m+k}^{m+n+k-1}(1-j\delta)}{(m+n)_n\prod_{j=m}^{m+n-1}(1-j\delta)}c_k^2\\
  &\qquad\quad+\sum_{k>m+n}(-1)^{n-1}\frac{(k-m-1)_n\prod_{j=m+k}^{m+n+k-1}(1-j\delta)}{(m+n)_n\prod_{j=m}^{m+n-1}(1-j\delta)}c_k^2.
 \end{split}
 \end{equation}

 The main result of this paper is contained in the following theorem.

 \begin{theo}\label{theo Poincare-type bounds}
 Let $X\sim\IP(\mu;\delta,\beta,\gamma)$ with $\delta\le0$. Fix two non-negative integers $m,n$ [with $n\ne0$] and a function $g\in\CH^{m,n}(X)$. Consider the quantity
 \begin{equation}\label{S_n}
 S_{m,n}(g)=\sum_{i=1}^{m}a_i\E^2q^i(X)g^{(i)}(X)+\sum_{i=1}^{n}(-1)^{i-1}b_i\E{q^{i}(X)\big(g^{(i)}(X)\big)^2},
 \end{equation}
 where
 \[
 a_i:=\frac{{m\choose{i}}\prod_{j=m+i}^{m+n+i-1}(1-j\delta)}{(m+n)_i\E{q^i(X)}\big(\prod_{j=i-1}^{2i-2}(1-j\delta)\big)\prod_{j=m}^{m+n-1}(1-j\delta)},
 \]
 \[
 b_i:=\frac{{n\choose{i}}}{(m+n)_i\prod_{j=m}^{m+i-1}(1-j\delta)}
 \]
 are strictly positive constants (depending only on $m,n$ and $X$) and the empty sums (when $m=0$) are treated as zero. Then the following inequality holds:
 \[
 (-1)^n\big(\Var{g(X)}-S_{m,n}(g)\big)\ge0,
 \]
 and where $S_{m,n}(g)$ becomes equal to $\Var{g(X)}$ if and only if $g$ is a polynomial of degree at most $m+n$.
 \end{theo}

 \begin{proof}
 From (\ref{S_n}), via (\ref{Eg^(i)(X)}) and (\ref{linear combination}), we obtain that $(-1)^n\big(\Var{g(X)}-S_{m,n}(g)\big)=R_{m,n}(g)$, where the residual
 \begin{equation}\label{(-1)^n(Var[g(X)-S_n]) 2}
 R_{m,n}(g)=\sum_{k>m+n}r_{k;m,n}(X)c_k^2
 :=\sum_{k>m+n}\frac{(k-m-1)_n\prod_{j=m+k}^{m+n+k-1}(1-j\delta)}{(m+n)_n\prod_{j=m}^{m+n-1}(1-j\delta)}c_k^2
 \end{equation}
 is non-negative and equals to zero if and only if $c_k=0$ for all $k>m+n$, i.e., the function $g:J(X)\to\R$ is a polynomial of degree at most $m+n$.
 \hfill$\square$
 \end{proof}
 \medskip

 Let $X\sim\IP(\mu;\delta,\beta,\gamma)$ with $\delta\le0$. Then $X$ is a linear function of a normal, a gamma or a beta random variable, see \cite{AP1}. The bounds $S_{m,n}(g)$ for the three main cases are included in Table \ref{table S_n}.

 \begin{table}[htb]
 \caption{Specific form of the bounds $S_{m,n}(g)$ for normal, gamma and beta distributions.}\label{table S_n}
 \begin{tabular}{@{\hspace{0ex}}lc@{\hspace{0ex}}ccc@{\hspace{0ex}}}
    \hline

    \hline
    \bf distribution & \bf \qquad parameters & \bf\em f & \bf{\em{J}}({\em{X}}) & \bf{\em{q}}({\em{x}}) \\
     \cline{2-5}
     & \multicolumn{4}{c}{\bf bounds $\textrm{\em S}_{\textrm{{\em m{\bf,}n}}}\textrm{({\em g})}$} \\
    \hline

    \hline\\
    [-2.2ex]
    Normal$(\mu,\sigma^2)$ & $\qquad\mu\in\R$, $\sigma^2>0$ & $\frac{1}{\sqrt{2\pi}\sigma}e^{-(x-\mu)^2/2\sigma^2}$ & $\R$ & $\sigma^2$ \\
    [2.5ex]
     &\multicolumn{4}{r}{\footnotesize$\sum_{i=1}^{m}\frac{{m\choose{i}}\sigma^{2i}}{(m+n)_i}\E^2g^{(i)}(X)
                                       +\sum_{i=1}^{n}(-1)^{i-1}\frac{{n\choose{i}}\sigma^{2i}}{(m+n)_i}\E\big(g^{(i)}(X)\big)^2$}\\
    [1ex]
     \cline{2-5}\\
    [-2.2ex]
    Gamma$(\alpha,\lambda)$& $\qquad \alpha,\lambda>0$ & $\frac{\lambda^\alpha}{\varGamma(\alpha)}x^{\alpha-1}e^{-\lambda{x}}$ & $(0,\infty)$ & $x/\lambda$ \\
    [2.5ex]
     &\multicolumn{4}{r}{\footnotesize$\sum_{i=1}^{m}\frac{{m\choose{i}}}{(m+n)_i[\alpha]_i}\E^2X^ig^{(i)}(X)
                                       +\sum_{i=1}^{n}(-1)^{i-1}\frac{{n\choose{i}}}{(m+n)_i\lambda^i}\E{X^i\big(g^{(i)}(X)\big)^2}$}\\
    [1ex]
     \cline{2-5}\\
    [-2.2ex]
    Beta$(\alpha,\beta)$& $\qquad\alpha,\beta>0$ & $\frac{\varGamma(\alpha+\beta)}{\varGamma(\alpha)\varGamma(\beta)}x^{\alpha-1}(1-x)^{\beta-1}$ & $(0,1)$ & $x(1-x)\big/(\alpha+\beta)$ \\
    [2.5ex]
     &\multicolumn{4}{r} {\footnotesize$\sum_{i=1}^{m}\frac{{m\choose{i}}[\alpha+\beta+m+i]_n[\alpha+\beta]_{2i}} {(m+n)_i[\alpha]_i[\beta]_i[\alpha+\beta+i-1]_i[\alpha+\beta+m]_n}\E^2X^i(1-X)^ig^{(i)}(X)$}\\
      &\multicolumn{4}{r}{\footnotesize$+\sum_{i=1}^{n}(-1)^{i-1}\frac{{n\choose{i}}}{(m+n)_i[\alpha+\beta+m]_i}\E{X^i(1-X)^i\big(g^{(i)}(X)\big)^2}$}\\
    [1ex]
    \hline

    \hline
 \end{tabular}
 \end{table}

 \begin{rem}\label{rem universal}\rm
 (a) For fixed $n$ and for any function $g\in\CH^{M,n}(X)$, where $M$ can be finite or infinite, the variance bounds $\{S_{m,n}(g)\}_{m=0}^M$ are of the same type, i.e. upper bound when $n$ is odd and lower bound when $n$ is even.
 \smallskip

 \noindent
 (b) The bounds $\{S_{m,n}(g)\}_{m=0}^{n}$ require the same conditions on the function $g$, i.e., $g\in\CH^{n,n}(X)$.
 \end{rem}

 \begin{rem}\label{rem universal and existings}\rm
 (a) The bound $S_{1,1}(g)$ is the bound $S_{1,\rm(str)}$ of (\ref{S_n upper strong}).
 \smallskip

 \noindent
 (b) The bounds $S_{0,n}(g)$ are the bounds $S_n$ which are given by (\ref{existing bounds}). Also, for the special case $m=0$, $n=1$ observe that $S_{0,1}(g)={S_{1}}={S_{1,\rm(weak)}}$; see (\ref{S_n upper weak}).
 \smallskip

 \noindent
 (c) The results shown in Theorem \ref{theo Poincare-type bounds} apply to the special case where $n=0$ (note that the second sum is empty and is treated as zero). In this case the lower bound $S_{m,0}(g)$ is reduced to the one given by (\ref{Bessel}).
 \end{rem}
 \begin{rem}\label{rem infty}
 {\rm
 Regarding the conditions on the function $g$ of Theorem \ref{theo Poincare-type bounds} we note that $g\in\CH^{\max\{m,n\},n-1}(X)\smallsetminus\CH^{\max\{m,n\},n}(X)$ implies that the bound $S_{m,n}(g)$ is trivial, i.e., $+\infty$ when $n$ is odd and $-\infty$ when $n$ is even.}
 \end{rem}

 Now, we seek for upper bounds of the non-negative residual $R_{m,n}(g)$.

 \begin{prop}\label{prop bounds of residual}
 Assume the conditions of Theorem {\rm\ref{theo Poincare-type bounds}} and, further, suppose that $g\in\CH^{T,T}(X)$ for some $T\in\{n,\ldots,m+n+1\}$. Then the residual $R_{m,n}(g)$, given by {\rm(\ref{(-1)^n(Var[g(X)-S_n]) 2})}, is bounded above by
 \begin{equation}\label{E(g^(tau)(X))^2}
 u_{\tau}\E{q^\tau(X)}\big(g^{(\tau)}(X)\big)^2, \quad \tau=n,n+1,\ldots,T,
 \end{equation}
 where $u_{\tau}\!=u_{m,n,\tau}(X)\!:={\prod_{j=2m+n+1}^{2m+2n}\!(1-j\delta)}\big/\big[{{m+n\choose{n}}(m+n+1)_\tau\prod_{j=m}^{m+n+\tau-1}\!(1-j\delta)}\big]$.
 \end{prop}
 \begin{proof}
 Using (\ref{E(g^(i)(X))^2}) we write the quantity (\ref{E(g^(tau)(X))^2}) in the form $\sum_{k=\tau}^{\infty}\pi_{k;\tau}c_k^2$. Next, consider the sequence $\big\{w_{k;\tau}={\pi_{k;\tau}}/{r_{k;m,n}(X)}\big\}_{k=m+n+1}^{\infty}$, where $r_{k;m,n}(X)$ are the numbers given by (\ref{(-1)^n(Var[g(X)-S_n]) 2}), and observe that this sequence is increasing in $k$, with $w_{m+n+1;\tau}=1$.
 \hfill$\square$
 \end{proof}

 \noindent
 The upper bounds (when there are at least two) of the residual $R_{m,n}(g)$, given by (\ref{E(g^(tau)(X))^2}), are not comparable. For example, consider the functions $g_1=\varphi_{\tau}$ and $g_2=\varphi_{m+n+2}$ (both belong to $\CH^\infty(X)$), where $\varphi_{k}$ are the polynomials given by (\ref{varphi_k}), and observe that
 \smallskip

 \centerline{$u_\tau\E{q}^\tau(X)\big(g_1^{(\tau)}(X)\big)=u_\tau{\tau!}\prod_{j=k-2}^{2\tau-2}(1-j\delta)>0=u_{\tau+1}\E{q}^{\tau+1}(X)\big(g_1^{(\tau+1)}(X)\big)$}
 \smallskip

 \noindent
 and
 \smallskip

 \centerline{$\ds\frac{u_\tau\E{q}^\tau(X)\big(g_2^{(\tau)}(X)\big)}{u_{\tau+1}\E{q}^{\tau+1}(X)\big(g_2^{(\tau+1)}(X)\big)}
 =\frac{(m+n-\tau+1)\big(1-(m+n+\tau)\delta\big)}{(m+n-\tau+2)\big(1-(m+n+\tau+1)\delta\big)}<1$,}
 \smallskip

 \noindent
 since $\delta\le0$.

\section{Investigating the bounds $\textrm{\em S}_{\textrm{\em m{\bf,}n}}$ for fixed \em{n}}
 \setcounter{equation}{0} \label{sec3}
 Next, for $n$ fixed, we investigate the bounds $S_{m,n}(g)$ as $m$ increases. We compare the variance bounds $S_{n,n}(g)$ and $S_n$, given by (\ref{S_n}) [for $m=n$] and (\ref{existing bounds}), respectively. Also, we compare the new upper variance bounds $S_{n,1}(g)$ and $S_{n-1,1}(g)$ with the existing Chernoff-type upper variance bounds $S_{n,{\rm(str)}}$ and $S_{n,{\rm(weak)}}$, respectively; see (\ref{S_n upper strong}) and (\ref{S_n upper weak}).

 \begin{theo}\label{theo S_n,m1 vs S_n,m2}
 Let $X\sim\IP(\mu;\delta,\beta,\gamma)$ with $\delta\le0$. Fix the positive integer $n$ and consider a function $g\in\CH^{M,n}(X)$, where $M$ can be finite ($M\ge{n}$) or infinite. Then for each $m_1,m_2$ such that $0 \le m_1 < m_2 \le M$ the following inequality holds
 \begin{equation}\label{R_n,m1 vs R_n,m2}
 \big|\Var{g(X)}-{S}_{m_1,n}(g)\big|\ge\zeta_{m_1,m_2,n}(\delta)\big|\Var{g(X)}-{S}_{m_2,n}(g)\big|,
 \end{equation}
 where $\zeta_{m_1,m_2,n}(\delta):=\frac{(m_2+n)_n\prod_{j=m_2}^{m_2+n-1}(1-j\delta)}{(m_1+n)_n\prod_{j=m_1}^{m_1+n-1}(1-j\delta)}>1$. The equality holds if and only if the function $g:J(X)\to\R$ is a polynomial of degree at most $n+m_1$.
 \end{theo}
 \begin{proof}
 Consider the positive sequence
 $\left\{\zeta_{k;m_1,m_2,n}(\delta)={r_{k;m_1,n}(X)}\big/{r_{k;m_2,n}(X)}\right\}_{k>m_2+n}$,
 where $r_{k;m,n}(X)$ are given by (\ref{(-1)^n(Var[g(X)-S_n]) 2}). This sequence is decreasing in $k$. Specifically,
 \[
 \zeta_{k;m_1,m_2,n}(\delta)\searrow\zeta_{m_1,m_2,n}(\delta)\equiv\frac{(m_2+n)_n\prod_{j=m_2}^{m_2+n-1}(1-j\delta)}{(m_1+n)_n\prod_{j=m_1}^{m_1+n-1}(1-j\delta)},
 \quad\textrm{as} \ k\to\infty.
 \]
 Moreover, we observe that ${r_{k;m_1,n}(X)}>0$ and ${r_{k;m_2,n}(X)}=0$ for all $k=n+m_1+1,\ldots,n+m_2$. Therefore (\ref{R_n,m1 vs R_n,m2}) follows.

 We write $\big|\Var{g(X)}-{S}_{m_1,n}(g)\big|-\zeta_{m_1,m_2,n}(\delta)\big|\Var{g(X)}-{S}_{m_2,n}(g)\big|=\sum_{k>n+m_1}\theta_{k}c_k^2$ and we observe that $\theta_{k}>0$ for all $k$. Thus, the equality in (\ref{R_n,m1 vs R_n,m2}) holds if and only if $g$ is a polynomial of degree at most $n+m_1$.
 \hfill$\square$
 \end{proof}

 Notice that if $\delta<0$ then ${\prod_{j=m_2}^{m_2+n-1}(1-j\delta)}\big/{\prod_{j=m_1}^{m_1+n-1}(1-j\delta)}>1$ for each $n$ and $m_1<m_2$. Therefore, $\zeta_{m_1,m_2,n}(\delta)\ge\zeta_{m_1,m_2,n}(0)=(m_2+n)_n/(m_1+n)_n$, since $\delta\le0$.
 \begin{rem}\label{rem infty2}\rm
 Assume the conditions of Theorem \ref{theo S_n,m1 vs S_n,m2}.
 (a) In view of Remark \ref{rem universal}(a), the bounds $\{S_{m,n}(g)\}_{m=0}^M$ are of the same type. From (\ref{R_n,m1 vs R_n,m2}) it is follows that the bound ${S}_{m_2,n}(g)$ is better than the bound ${S}_{m_1,n}(g)$. Now, writing $n=2r$ (when $n$ is even) or $n=2r+1$ (when $n$ is odd) we observe that
 \[
 S_{0,2r}(g)\le S_{1,2r}(g)\le\cdots\le\Var{g}(X)\le\cdots\le S_{1,2r+1}(g)\le S_{0,2r+1}(g).
 \]
 (b) For the case $M=\infty$, from (\ref{Var g(X)}), (\ref{S_n}) and (a) it follows that
  \[
  \substack{\displaystyle S_{m,n}(g)\nearrow\Var{g(X)} \\ \mbox{\footnotesize[when $n$ is even]}}
  \ \ \substack{\displaystyle \textrm{or} \\ \ } \ \
  \substack{\displaystyle S_{m,n}(g)\searrow\Var{g(X)},\\ \mbox{\footnotesize[when $n$ is odd]}}
  \quad
  \substack{\displaystyle \textrm{as}\quad m\to\infty. \\ \ }
  \]
 \end{rem}

 Now, we compare the existing variance bounds $S_n \big(\equiv{S_{0,n}(g)}\big)$ with the best proposed bound shown in this article requiring the same conditions on $g$, i.e., with the bound $S_{n,n}(g)$; see Remark \ref{rem universal}(b).
 \begin{cor}\label{cor S_n vs widetilde(S)_n}
 The variance bounds $S_{n,n}(g)$ and $S_n$, given by {\rm(\ref{S_n})} (for $m=n$) and {\rm(\ref{existing bounds})} respectively, are of the same type and require the same conditions on $g$. Moreover, the new bound $S_{n,n}(g)$ is better than the existing bound $S_n$. Specifically,
 \[
 \big|\Var{g(X)}-S_{n}\big|\ge{2n\choose{n}}\frac{\prod_{j=n}^{2n-1}(1-j\delta)}{\prod_{j=0}^{n-1}(1-j\delta)}\big|\Var{g(X)}-{S}_{n,n}(g)\big|.
 \]
 The equality holds only in the trivial cases when $\Var{g(X)}=S_{n,n}(g)=S_{n}$, i.e., the function $g:J(X)\to\R$ is a polynomial of degree at most $n$.
 \newline
 Note that, since $\delta\le0$, ${2n\choose{n}}{\prod_{j=n}^{2n-1}(1-j\delta)}\big/{\prod_{j=0}^{n-1}(1-j\delta)}\ge{2n\choose{n}}$.
 \end{cor}

 The quantities $S_{n,\rm(str)}$ and $S_{n,1}(g)$ are upper bounds for $\Var{g(X)}$. Both bounds are equal to $\Var{g(X)}$ if and only if the function $g$ is a polynomial of degree at most $n+1$. The quantities $S_{n,\rm(weak)}$ and $S_{n-1,1}(g)$ are upper bounds for $\Var{g(X)}$. Both bounds are equal to $\Var{g(X)}$ if and only if the function $g$ is a polynomial of degree at most $n$. Thus, it is reasonable to compare these bounds.
 \begin{theo}\label{theo S_n,(weak)-S_1,n-1 and S_n,(strong)-S_1,n}
 For $n=1,2,\ldots$ and any suitable function $g$ we have that:
 \smallskip

 \noindent
 {\rm(a)} ${S_{n,1}(g)}\le{S_{n,\rm(str)}}$, where the equality holds when $n=1$ or $n>1$ and $g$ is a polynomial of degree at most $n+1$.
  \smallskip

 \noindent
 {\rm(b)} ${S_{n-1,1}(g)}\le{S_{n,\rm(weak)}}$, where the equality holds when $n=1$ or $n>1$ and $g$ is a polynomial of degree at most $n$.
 \end{theo}
 \begin{proof}
 (a) From (\ref{S_n}) and (\ref{S_n upper strong}), via (\ref{E(g^(i)(X))^2}) and (\ref{Eg^(i)(X)}), we have that
 \begin{equation}\label{S_n,(strong)-S_1,n}
 \mbox{$S_{n,\rm(str)}-S_{n,1}(g)=\sum_{k>n+1}\frac{k[1-(k-1)\delta]}{(n+1)(1-n\delta)}
                            \Big(\frac{{k-1\choose n-1}\prod_{j=k}^{n+k-2}(1-j\delta)}{n\prod_{j=n+1}^{2n-1}(1-j\delta)}-1\Big)c_k^2,$}
 \end{equation}
 where each coefficient of $c_k^2$, $k>n+1$, is zero when $n=1$ and is positive when $n>1$.

 \noindent
 (b) Similarly, from (\ref{S_n}) and (\ref{S_n upper weak}), via (\ref{E(g^(i)(X))^2}) and (\ref{Eg^(i)(X)}), it follows that
 \begin{equation}\label{S_n,(weak)-S_1,n-1}
 \mbox{$S_{n,\rm(weak)}-S_{n-1,1}(g)=\sum_{k>n}\frac{k[1-(k-1)\delta]}{n[1-(n-1)\delta]}
                              \Big(\frac{{k-1\choose n-1}\prod_{j=k}^{n+k-2}(1-j\delta)}{\prod_{j=n}^{2n-2}(1-j\delta)}-1\Big)c_k^2,$}
 \end{equation}
 where each coefficient of $c_k^2$, $k>n$, is zero when $n=1$ and is positive when $n>1$.
 \hfill$\square$
 \end{proof}
 \begin{rem}\label{rem S_n,(weak)-S_1,n-1 and S_n,(strong)-S_1,n}
 \rm
 For each $n=2,3,\ldots$ it follows that:
 \smallskip

 \noindent
 (a) The bound ${S_{n,1}(g)}$ is better than the bound ${S_{n,\rm(str)}}$; notice that the bound $S_{n,1}(g)$ requires a milder finiteness condition, $g\in\CH^{n,1}(X)$, compared to $S_{n,\rm(str)}$ which requires that $g\in\CH^{n,n}(X)\subseteq\CH^{n,1}(X)$.
 \smallskip

 \noindent
 (b) The bound ${S_{n-1,1}(g)}$ is better than the bound ${S_{n,\rm(weak)}}$; as in (a), the bound $S_{n-1,1}(g)$ requires a weaker finiteness condition, i.e. $g\in\CH^{n-1,1}(X)$, rather than $S_{n,\rm(weak)}$, i.e. $g\in\CH^{n,n}(X)\subseteq\CH^{n-1,1}(X)$.
 \end{rem}

 \noindent
 {\bf Final Conclusion} \ {\it The variance bounds given by Theorem {\rm\ref{theo Poincare-type bounds}}, for appropriate choices of $n$ and $m$, either provide existing univariate variance bounds or improvements. Our bounds cover all usual cases, namely:
 {\rm
 \begin{itemize}
 \item Chernoff-type

       [\cite{Nash,BL,Cher,CP1,PP,P-Rao,AP2}],
 \item Poincar\'e-type

       [\cite{Pap1,Cac2,Joh,HK,APP}],
 \item Bessel-type

       [\cite{Cac,HK,APP}].
 \end{itemize}}
 \noindent
 Note that no further conditions on the function $g$ are imposed; instead, the new bounds require the same or weaker conditions, see Remarks {\rm\ref{rem universal}}, {\rm\ref{rem universal and existings}} and {\rm\ref{rem S_n,(weak)-S_1,n-1 and S_n,(strong)-S_1,n}}, Theorem {\rm\ref{theo S_n,(weak)-S_1,n-1 and S_n,(strong)-S_1,n}}, and Corollary {\rm\ref{cor S_n vs widetilde(S)_n}}. Therefore, the new proposed variance bounds outweigh all the existing variance bounds presented in the bibliography.}

 \begin{acknowledgements}
 From this position I would like to thank Professor N. Papadatos for his helpful observations and comments. I would also like to thank an anonymous Associate Editor who carefully read the revised manuscript and kindly brought to my attention a typing error in the proof of Lemma \ref{lem E q^m(X) |g^(m)(X)|<00 => E q^i(X) |g^(i)(X)|<00}.
 \end{acknowledgements}
 \vspace{3.9ex}

 \begin{appendix}
 \noindent{\large\bf Appendix}
 \vspace{-2em}
 \section{Some useful properties of the Integrated Pearson Family}
 \label{appendix properties}
 \setcounter{equation}{0}
 The following properties have been reproduced from \cite{APP,AP1,AP2} and are stated here for easy reference.
 \smallskip

 \noindent
 Consider a random variable $X$ with density $f\sim\IP(\mu;q)\equiv\IP(\mu;\delta,\beta,\gamma)$.

 Let $a\in(1,+\infty)$. Then $E|X|^a<\infty$ if and only if $\delta<1/(a-1)$. Notice that $X$ has finite moments of any order if and only if $\delta\le0$; see
 \cite[Corollary 2.2]{AP1}.

 The support of $X$ is the interval $J(X)=(\alpha,\omega)$ and the density $f\in C^\infty(\alpha,\omega)$, see \cite{AP1}.

 If $\E|X|^{2i+1}<\infty$, $i\in\N^*\equiv\N\smallsetminus\{0\}$ (that is, $\delta<1/2i$), then the random variable $X_i$ with density function $f_i(x)=q^i(x)f(x)/\E q^i(X)$ follows $\IP(\mu_i;q_i)$ distribution with $\mu_i=(\mu+i\beta)/(1-2i\delta)$ and $q_i=q/(1-2i\delta)$; see \cite[Theorem 5.2]{AP1}. Note that if $\E|X|^{2i}<\infty$ and $\E|X|^{2i+1}=\infty$ then the function $f_i$ is a probability density function; however, $\E|X_i|=\infty$ so $X_i$ does not belong to the Integrated Pearson Family.

 If $\E|X|^{2N}<\infty$, $N\in\N^*$ (that is, $\delta<1/(2N-1)$), then the quadratic $q$ generates the orthogonal polynomials through the Rodrigues-type formula, \cite{Pap2},
 \[
 P_k(x)=\frac{(-1)^k}{f(x)}\frac{\ud^k}{\ud x^k}\big[q^k(x)f(x)\big],\quad x\in J(X),\quad k=0,1,\ldots,N.
 \]
 \cite{APP} showed an extended Stein-type identity of order $n$. This identity takes the form $\E{P_k(X)g(X)}=\E{q^k(X)g^{(k)}(X)}$, provided that $\E{q^k(X)|g^{(k)}(X)}|<\infty$. Also, $\E{P_k(X)P_m(X)}=\delta_{k,m}k!\E{q^k(X)}\times\prod_{j=k-1}^{2k-2}(1-j\delta)$; thus, the system of polynomials $\{\varphi_k\}_{k=0}^N$ is orthonormal, with respect to density $f$, where
 \begin{equation}\label{varphi_k}
 \varphi_k(x)=P_k(x)\Big(k!\E{q^k(X)}\prod_{j=k-1}^{2k-2}(1-j\delta)\Big)^{-1/2}.
 \end{equation}
 Hence,
 \begin{equation}\label{E q^k(X)g^(k)(X)}
 \E q^k(X)g^{(k)}(X)=\Big(k!\E{q^k(X)}\prod_{j=k-1}^{2k-2}(1-j\delta)\Big)^{-1/2}\E\varphi_k(X)g(X).
 \end{equation}
 Moreover, the system of polynomials $\big\{\varphi^{(i)}_{k+i}\big\}_{k=0}^{N-i}$ [where $\varphi_k$ are the polynomials given by (\ref{varphi_k}) and $\varphi^{(i)}_{k+i}$ is the $i$-th derivative of $\varphi_{k+i}$] is orthogonal with respect to density $f_i$. Specifically, if $\varphi_{k,i}$ are the orthonormal  polynomials corresponding to the density $f_i$ then $\varphi^{(i)}_{k+i}(x)=\nu^{(i)}_{k}\varphi_{k,i}(x)$, where $\nu^{(i)}_{k}=\nu^{(i)}_{k}(X):=\big[(k+i)_i\times\big(\prod_{j=k+i-1}^{k+2i-2}(1-j\delta)\big)/\E{q^i}(X)\big]^{1/2}$, see \cite[Corollary 5.4]{AP1}.
 \begin{lem}\label{lem E q^i|g^(i)|}
 {\rm [\cite[Theo. 3.1(b), p. 516]{APP}]}\quad
 Let $X\sim\IP(\mu;q)$, with $\E X^{2k}<\infty$, and a suitable function $g$, with $\E q^k(X)|g^{(k)}(X)|<\infty$, then $\E|P_k(X)g(X)|<\infty$.
 \end{lem}
 \begin{lem}\label{lem E q^i(g^(i))^2}
 {\rm [\cite{AP2}]}\quad
 Let a random variable $X\sim\IP(\mu;q)$ and consider the strictly positive integers $n$ and $N$ such that $n\le N$ and $\E|X|^{2N}<\infty$.
 \begin{equation}\label{E q^n(X)g^(n)(X)<inf => E q^i(X)g^(i)(X)<inf}
 \textrm{If} \ g\in\CH^{n,n}(X) \ \textrm{then} \ g^{(i)}\in\CH^{n-i,n-i}(X_i)\quad \textrm{for each}\ i=0,1,\ldots n-1.
 \end{equation}
 \begin{equation}\label{E varphi_k,i(X_i)g^(i)(X_i)}
 \E\varphi_{k,i}(X_i)g^{(i)}(X_i)=\nu_k^{(i)}(X)\E\varphi_{k+i}(X)g(X)\quad \textrm{for each}
 \left\{
 \begin{array}{l}
   i=1,2,\ldots,n, \\
   k=0,1,\ldots,N-i.
 \end{array}
 \right.
 \end{equation}
 \end{lem}

  If the parameter $\delta$ of $q$ is non-positive then the moment generating function of $X$ is finite in a neighborhood of zero; thus the system of polynomials $\{\varphi_k\}_{k=0}^{\infty}$ forms an orthonomal basis of $L^2(\R,X)$ and the Parseval identity holds, see \cite{AP1}. Notice that for each $i\in\N$ the  parameter $\delta_i=\frac{\delta}{1-2i\delta}$ is non-positive too. Thus, the system of polynomials $\{\varphi_{k,i}\}_{k=0}^{\infty}$ is an orthonomal basis of $L^2(\R,X_i)$ and the Parseval identity holds too.

 \section{The solution of the system (\ref{system of equations})}
 \label{appendix system}
 \setcounter{equation}{0}
 Consider the determinants $d_{m,n}=\det(A_{m,n})$ and $d_{i;m,n}=\det(A_{i;m,n})$, $i=1,2,\ldots,n$, where the matrix $A_{i;m,n}$ is formed from $A_{m,n}$ by replacing column $i$ with the vector $\mathbf{1}_n$.

 For each $t=1,2,\ldots,n$ define the matrix $B_{m,n}(t)\in\R^{(n-t+1)\times(n-t+1)}$ [$m,n$ are fixed] which has elements $b_{r,c;m,n}(t)=(m+r-1)_{c-1}\prod_{j=m+r+t-1}^{m+r+c+t-3}(1-j\delta)$, where empty products are treated as one. Observe that

 \noindent
 $d_{m,n}={(m+n)_n\big(\prod_{j=m}^{m+n-1}(1-j\delta)\big)}\det\big(B_{m,n}(1)\big)$ and

 \noindent
 $\det\big(B_{m,n}(t)\big)=(n-t)!\big(\prod_{j=m+1}^{m+n-t}(1-[2j+(t-1)]\delta)\big)\det\big(B_{m,n}(t+1)\big)$, $t=1,2,\ldots,n-1$.

 \noindent
 Thus, it follows that

 \centerline{$d_{m,n}=(m+n)_n\big[\prod_{j=0}^{n-1}j!\big]\big[\prod_{j=m}^{m+n-1}(1-j\delta)\big]\prod_{t=1}^{n-1}\!\!\big(\prod_{j=m+1}^{m+n-t}(1-[2j+(t-1)]\delta)\big)\ne0$.}

 Now, for each $t=1,2,\ldots,n$ define the matrix $B_{i,m,n}(t)\in\R^{(n-t+1)\times(n-t+1)}$ [$i,m,n$ are fixed integers] which has $(r,c)$-element $b_{r,c;i,m,n}(t)=(m+r)_{c-1}\prod_{j=m+r+t-2}^{m+r+c+t-4}(1-j\delta)$, when $c=1,2,\ldots,i-t$, and $b_{r,c;i,m,n}(t)=(m+r)_{c}\prod_{j=m+r+t-2}^{m+r+c+t-3}(1-j\delta)$, when $c=i-t+1,i-t+2,\ldots,n-t+1$. Observe that

 \noindent
 $d_{i;m,n}=(-1)^{i-1}\det\big(B_{i,m,n}(1)\big)$,

 \noindent
 $\det\big(B_{i,m,n}(t)\big)=\frac{(n-t+1)!}{(i-t+1)}\big(\prod_{j=m+1}^{m+n-t}(1-[2j+(t-1)]\delta)\big)\det\big(B_{i,m,n}(t+1)\big)$, $t=1,2,\ldots,i-1$, and

 \noindent
 $\det\big(B_{i,m,n}\!(i)\big)=(n-i+1)!\big(\!\prod_{j=m+1}^{m+n-i}(1-[2j+(i-1)]\delta)\!\big)(m+n-i)_{n-i}\big(\!\prod_{j=m+i}^{m+n-1}\!(1-j\delta)\!\big)\det\big(B_{m,n}\!(i+1)\big)$.

 \noindent
 Thus, it follows that

 \centerline{$d_{i;m,n}=(-1)^{i-1}\frac{(m+n-i)_{n-i}}{i!(n-i)!}\big[\prod_{j=0}^{n}j!\big]\big[\prod_{j=m+i}^{m+n-1}(1-j\delta)\big]\prod_{t=1}^{n-1}\big(\prod_{j=m+1}^{m+n-t}(1-[2j+(t-1)]\delta)\big)$.}

 Therefore, according to Cram\'er's rule, (\ref{lambda_i,n;0}) follows.

 \section{A useful hypergeometric series}
 \label{appendix hypergeometric}
 \setcounter{equation}{0}
 \begin{lem}\label{lem hypergeometric}
 Let $m,n,k\in\N$ and $\delta\le0$. Then the following hypergeometric series holds:
 \begin{equation}\label{Hypergeometric}
 \sum_{i=0}^{n}(-1)^{i}{n\choose{i}}\frac{(k)_i}{(m+n)_i}\frac{\prod_{j=k-1}^{k+i-2}(1-j\delta)}{\prod_{j=m}^{m+i-1}(1-j\delta)}
 =\frac{(m+n-k)_n\prod_{j=m+k}^{m+n+k-1}(1-j\delta)}{(m+n)_n\prod_{j=m}^{m+n-1}(1-j\delta)}.
 \end{equation}
 \end{lem}
 \begin{proof}
 For the case $\delta=0$, write (\ref{Hypergeometric}) as
 \[
 \sum_{i=0}^{n}(-1)^{i}{n\choose{i}}\frac{(k)_i}{(m+n)_i}=\frac{(m+n-k)_n}{(m+n)_n}
 \]
 and observe that this follows from Vandermonde's formula,
 \[
 \sum_{i=0}^{n}(-1)^i{n\choose{i}}\frac{(x)_i}{(x+y)_i}=\frac{(y)_n}{(x+y)_n},
 \]
 by replacing $x$ with $k$ and $y$ with $m+n-k,$ see \cite[p. 125]{Char}.
 \smallskip

 \noindent
 For the case $\delta<0$, write (\ref{Hypergeometric}) as
 \[
 \sum_{i=0}^{n}(-1)^{i}\frac{(n)_i(k)_i(1/\delta+1-k)_i}{i!(m+n)_i(1/\delta-m)_i}=\frac{(m+n-k)_n(1/\delta-m-k)_n}{(m+n)_n(1/\delta-m)_n}.
 \]
 This follows from Dougall's identity,
 \[
 \sum_{i=0}^s\frac{(\alpha)_i(\beta)_i(s)_i}{i![\gamma+1]_i(\alpha+\beta+\gamma+s)_i} =\frac{[\alpha+\gamma+1]_s[\beta+\gamma+1]_s}{[\gamma+1]_s[\alpha+\beta+\gamma+1]_s},
 \]
 using the substitution $\alpha\mapsto k$, $\beta\mapsto(1/\delta+1-k)$, $\gamma\mapsto(-m-n-1)$ and $s\mapsto n$, see \cite[eq. (2)]{Dougall}.
 \hfill$\square$
 \end{proof}

\end{appendix}
\end{document}